\documentclass[preprint,12pt]{elsarticle}
\usepackage{amsthm,amsmath,amssymb}
\usepackage[colorlinks=true,citecolor=black,linkcolor=black,urlcolor=blue]{hyperref}
\usepackage{graphicx}
\usepackage{float}
\usepackage{mathrsfs}
\usepackage{mathtools}
\usepackage{algorithm}
\usepackage{algpseudocode}
\usepackage{comment}
\usepackage{epstopdf}
\usepackage{epsfig}
\usepackage{subfigure}
\newcommand{\Break}{\State \textbf{break}}
\newtheorem*{claim*}{Claim}
\usepackage{tikz}
\usetikzlibrary{positioning}
\usepackage{float}
\usetikzlibrary{shapes.geometric} 
\usepackage{caption}

\theoremstyle{plain}
\newtheorem{theorem}{Theorem}
\newtheorem{lemma}[theorem]{Lemma}
\newtheorem{corollary}[theorem]{Corollary}
\newtheorem{proposition}[theorem]{Proposition}

\theoremstyle{definition}
\newtheorem{definition}[theorem]{Definition}

\newtheorem{problem}[theorem]{Problem}
\newtheorem{question}[theorem]{Question}
\theoremstyle{remark}
\newtheorem*{remark}{Remark}

\title{Resolving problems on polynomial characterizations of daisy cubes
and extensions}

\author{Xuan Zheng, Yan-Ting Xie,
Shou-Jun Xu$^*$\\
\small School of Mathematics and Statistics, Gansu Center for Applied Mathematics, Lanzhou University, Lanzhou, Gansu 730000, China\\}
\begin{document}
\begin{abstract}
Let $X\subseteq\{0,1\}^n$ be a set of binary strings of length $n$.
The daisy cube $Q_n(X)$ is the subgraph of the hypercube $Q_n$ induced by the union of the intervals $I(0^n,x)$ for $x\in X$.
As a subclass of partial cubes, it generalizes Fibonacci cubes and Lucas cubes.
For a graph $G$ and a vertex $u\in V(G)$, the generating function of the number of $k$-cubes (resp. $k$-cubes at distance $d$ from $u$, and vertices at distance $d$ from $u$) is called the cube polynomial $C_G(x)$ (resp. the distance cube polynomial $D_{G,u}(x,y)$, and the distance polynomial $W_{G,u}(x)$).

Let $G$ be a partial cube embedded into the hypercube $Q_n$ with $0^n \in V(G)$. In this paper, we prove that $G$ is a daisy cube if and only if one of the following equivalent conditions holds: (1) $C_{G}(x)=W_{G,0^n}(x+1)$; (2) $D_{G,0^n}(x,y)=W_{G,0^n}(x+y)$; (3) $D_{G,0^n}(x,y)=C_{G}(x+y-1)$. In particular, the results related to (1) and (3) give affirmative answers
to two open problems posed by Klav\v{z}ar and Mollard (2019).
Meanwhile, our results yield non-constructive characterizations of daisy cubes, which answer the question posed by Taranenko (2020).
Further, we prove that $D_{G, u}(x, y)\leq W_{G, u}(x+y)$ and $C_{G}(x)\leq W_{G, u}(x+1)$ among the whole class of partial cubes. Besides, combined with another sharp upper bound $Cl_{G^\#}(x+1)$ for $C_G(x)$ due to Xie et al. (2024),
we obtain polynomial characterizations of simplex graphs (a subclass of daisy cubes): $G$ is a simplex graph if and only if $W_{G, 0^n}(x)=Cl_{G^\#}(x)$, here  $Cl_{G^\#}(x)$ is the clique polynomial  of the crossing graph $G^\#$ of $G$.
\end{abstract}
\maketitle
\noindent
$^*$Corresponding author, e-mail: \texttt{shjxu@lzu.edu.cn}

\section{Introduction}
The \textit{$n$-dimensional hypercube} (or \textit{$n$-cube}) $Q_n$ is the graph with vertex set
\[
V(Q_n) \,=\, B^n \,:=\, \big\{ u_1u_2\cdots u_n \mid u_i\in\{0,1\},\ 1\le i\le n \big\}.
\]
Two vertices are adjacent if and only if the corresponding binary strings differ in precisely one position. Throughout this paper, $V(G)$ denotes the vertex set of a graph $G$.
A \textit{partial cube} is defined as an isometric subgraph of a hypercube. Let $0^n$ be the all-zero vertex of $Q_n$, and let $X\subseteq V(Q_n)$ be a nonempty subset.
The \textit{daisy cube} $Q_n(X)$ generated by $X$, introduced by Klav\v zar and Mollard~\cite{SKM},
is the induced subgraph of $Q_n$ spanned by all vertices lying on some shortest path from $0^n$ to a vertex in $X$.
Equivalently, let $\preceq$ be the partial order on $B^n$ such that
$u_1u_2\cdots u_n \preceq v_1v_2\cdots v_n$ holds whenever $u_i\le v_i$ for all $i\in [n]$.
Then
\[
V(Q_n(X)) \,=\, \bigl\{ u\in B^n \mid u\preceq x \text{ for some } x\in X \bigr\}.
\]

In combinatorics, a set system $\Delta$ is called an {\em abstract simplicial complex}
(or an {\em independence system}) if it is closed under taking subsets;
that is, if $X\in\Delta$ and $Y\subseteq X$, then $Y\in\Delta$.
There is a natural bijection between $B^n$ and the power set of $[n]=\{1, 2, \cdots, n\}$,
where each vertex $u=u_1u_2\cdots u_n\in B^n$ corresponds to the set
$\{i\in[n]: u_i=1\}$.
Via this correspondence, the vertex set of any daisy cube $Q_n(X)$
forms an abstract simplicial complex.
Hence, daisy cubes can be viewed as graph-theoretic representations of abstract simplicial complexes. 

Daisy cubes contain several important families as subclasses,
such as Fibonacci cubes~\cite{k13}, Lucas cubes~\cite{mcs01,t13},
and resonance graphs of kinky benzenoid systems in chemical graph theory~\cite{z18}.

Klav\v{z}ar and Mollard~\cite{SKM} raised problems of characterizing daisy cubes.
Taranenko~\cite{t20} gave a characterization of daisy cubes in terms of an expansion procedure. Whereas we consider the characterization of daisy cubes via identities involving graph polynomials, described as follows.
We proceed to recall the definitions of the relevant graph polynomials.

\begin{definition}[\cite{BSKR1}]\label{def01}
For a graph $G$, let $c_k(G)$ ($k\geq 0$) be the number of induced subgraphs of $G$ isomorphic to $Q_k$ in $G$. The \textit{cube polynomial} of $G$, $C_G(x)$, is defined as
    \[C_G(x)=\sum_{k\ge0}c_k(G)x^k.\]
\end{definition}

The cube polynomial was first formally introduced by Brešar, Klavžar, and Škrekovski \cite{BSKR1} and further studied in \cite{BSKR2,xfx24}.

\begin{definition}[\cite{SKM}]\label{def03}
For a graph $G$ and $u\in V(G)$, let $w_d(G)$ ($d \ge 0$) be the number of vertices of $G$ at distance $d$ from $u$. The \textit{distance polynomial} of $G$ with respect to $u$, $W_{G,u}(x)$, is defined as
    \[W_{G,u}(x)=\sum_{d\ge0}w_d(G)x^d.\]
\end{definition}

\begin{definition}[\cite{SKM}]\label{def02}
For a graph $G$ and $u\in V(G)$, let $c_{k,d}(G)$ ($k,d \geq 0$) be the number of induced subgraphs of $G$ isomorphic to $Q_k$ at distance $d$ from $u$. The \textit{distance cube polynomial} of $G$ with respect to $u$, $D_{G,u}(x,y)$, is defined as
\[D_{G,u}(x,y) = \sum_{k,d \geq 0} c_{k,d}(G) x^k y^d.\]
\end{definition}

By the definitions, we can derive that the distance cube polynomial is the refinement of both the cube polynomial and the distance polynomial, i.e., for any graph (not necessary partial cube) $G$,  $D_{G,u}(0,y)=W_{G,u}(y)$ and
\begin{align}
    D_{G,u}(x,1)=C_G(x). \label{eq:D=C}
\end{align}

Klav\v{z}ar and Mollard in \cite{SKM} proved the following results.
\begin{proposition}{\em\cite{SKM}}\label{prop01}
If a partial cube $G$ is a daisy cube, then
\begin{enumerate}
\item[{\em (a)}]\label{eq:a}$C_{G}(x)=W_{G,0^n}(x+1)$; 
\item[{\em (b)}]$D_{G,0^n}(x,y)=W_{G,0^n}(x+y)$; 
\item[{\em (c)}]$D_{G,0^n}(x,y)=C_{G}(x+y-1)$. 
\end{enumerate}
\end{proposition}

Furthermore, they posed the following problem:
\begin{problem}(Problems 5.2 and 5.3 in \cite{SKM})\label{prob01}
Does the equality (a) or (c) imply that $G$ is a daisy cube?
\end{problem}

In the present paper, we resolve this problem completely and the additional problem related to (b). Meanwhile, our results yield non-constructive characterizations of daisy cubes, which answer the question posed by Taranenko in \cite{t20}.

\begin{theorem}\label{the01}
Let $G$ be a partial cube embedded into the hypercube $Q_n$ with $0^n \in V(G)$. Then the following assertions are equivalent:
\begin{enumerate}[(i)]
    \item $G$ is a daisy cube;
    \item $C_{G}(x)=W_{G,0^n}(x+1)$;
    \item $D_{G,0^n}(x,y)=W_{G,0^n}(x+y)$;
    \item $D_{G,0^n}(x,y)=C_{G}(x+y-1)$.
\end{enumerate}
\end{theorem}
\begin{remark}
The additional condition that $0^n\in V(G)$ does not weaken Theorem \ref{the01}.  Indeed, if $0^n\not\in V(G)$, pick an arbitrary vertex $u\in V(G)$. We may relabel each vertex $v\in V(G)$ as $v\bigtriangleup u$, which sends $u$ to $0^n$. Here $v\bigtriangleup u$ is defined by $(v\bigtriangleup u)_i=v_i+u_i$ in the modulo 2 for $i\in [n]$. 
\end{remark}
If we generalize our study from the class of daisy cubes to the class of all partial cubes, what relations hold between the graph polynomials $C_G(x)$, $D_{G,u}(x,y)$, and $W_G(x)$ of a partial cube $G$?

Let $P(x_1,x_2,\cdots,x_n)$ and $Q(x_1,x_2,\cdots,x_n)$ be two polynomials with same variables. We define $P(x_1,x_2,\cdots,x_n)\leq Q(x_1,x_2,\cdots,x_n)$ component-wise: each coefficient of $P(x_1,x_2,\cdots, x_n)$ is no greater than the corresponding coefficient of $Q(x_1,  x_2, \cdots, x_n)$. If strict inequality holds for as least one coefficient, we write $P(x_1, x_2, \cdots, x_n)< Q(x_1,x_2,\cdots,x_n)$.  If all are equal, we write $P(x_1, x_2, \cdots, x_n)= Q(x_1,x_2,\cdots,x_n)$. Our second main result in this paper is (combined with Theorem \ref{the01}): 
\begin{theorem}\label{the02}
If $G$ is a partial cube with $u\in V(G)$, then $D_{G, u}(x, y) \leq W_{G,u}(x+y)$,  where equality holds if and only if $G$ is a daisy cube and $u=0^n$.
\end{theorem}

Substituting $y=1$, combined with Eq. \ref{eq:D=C}, we obtain the following corollary immediately.
\begin{corollary}\label{cor01}
If $G$ is a partial cube with $u\in V(G)$, then $C_{G}(x) \leq W_{G,u}(x+1)$, the equality holds if and only if $G$ is a daisy cube and $u=0^n$.
\end{corollary}

Alternatively, Corollary \ref{cor01} provides an upper bound for $C_{G}(x)$. In fact, Xie et al. \cite{xfx24} derived another upper bound $Cl_{G^{\#}}(x+1)$, which is the clique polynomial of the crossing graph $G^\#$ of $G$, as follows.
\begin{theorem}{\em\cite{xfx24}}\label{prop02}
Let $G$ be a partial cube and $G\neq K_1$. Then
\begin{equation}\label{ine:crossing}
C_G(x)\leq Cl_{G^{\#}}(x+1) 
\end{equation}
and the equality holds if and only if $G$ is a median graph.
\end{theorem}

For the two upper bounds $W_{G, u}(x+1)$ and $Cl_{G^{\#}}(x+1)$ on $C_G(x)$,  it is natural to investigate their relationship. Recently, two of the authors of this paper \cite{xx25} characterized the relation between daisy cubes, median graphs, and simplex graphs (the definition will be introduced in the next section), as follows.

\begin{lemma}{\em\cite{xx25}}\label{lem04}
Let $G$ be a partial cube. Then $G$ is a simplex graph if and only if $G$ is a median graph and a daisy cube simultaneously. 
\end{lemma}

The discussions above lead to two topics: the polynomial characterization of simplex graphs, and the comparison of $W_{G, u}(x+1)$ with $Cl_{G^{\#}}(x+1)$. First, according to Corollary \ref{cor01} and Theorem \ref{prop02}, it is obvious that
\begin{align}
    W_{G,0^n}(x+1)=C_G(x)=Cl_{G^{\#}}(x+1) \label{eq:W=C=Cl}
\end{align}
if and only if $G$ is exactly the intersection of daisy cube and median graph, i.e., $G$ is a simplex graph. If we omit $C_G(x)$ in Eq. \eqref{eq:W=C=Cl} and only consider the two upper bounds for $C_G(x)$, i.e., $Cl_{G^{\#}}(x+1)=W_{G,0^n}(x+1)$, is the above if-and-only-if statement still true? One direction is obvious: if $G$ is a simplex graph, then $Cl_{G^{\#}}(x+1)=W_{G,0^n}(x+1)$.  We prove that the converse direction still holds, i.e., the third main result, as follows.   

\begin{theorem}\label{the04}
Let $G$ be a partial cube embedded into the hypercube $Q_n$ with $0^n \in V(G)$ and $G\neq K_1$. $G$ is a simplex graph if and only if $W_{G,0^n}(x)=Cl_{G^{\#}}(x)$, here $G^{\#}$ is the crossing graph of $G$.
\end{theorem}

Second, in contrast to Corollary \ref{cor01} and Theorem \ref{prop02}, one may ask whether there exists comparability relation between $W_{G, u}(x+1)$ and $Cl_{G^{\#}}(x+1)$? We give a negative answer. If $G$ is a median graph but not a daisy cube, by  Corollary \ref{cor01} and Theorem \ref{prop02}, $Cl_{G^{\#}}(x+1)=C_G(x)<W_{G, u}(x+1)$. Similarly, if $G$ is a daisy cube with $u=0^n$ but not a median graph, $W_{G, u}(x+1)=C_G(x)<Cl_{G^{\#}}(x+1)$. For examples (see Figure \ref{fig:example}), the 3-path $P_4$ is a median graph but not a daisy cube, then 
\begin{align*}
Cl_{(P_4)^{\#}}(x+1)&=3x+4,\\
W_{P_4,u}(x+1)&=x^3+4x^2+6x+4;    
\end{align*}
$Q_3^-$ (the graph obtained by deleting one vertex from $Q_3$) is a daisy cube but not a median graph, then
\begin{align*}
Cl_{(Q_3^-)^{\#}}(x+1)&=x^3+6x^2+12x+8,\\
W_{Q_3^-,u}(x+1)&=3x^2+9x+7.    
\end{align*}
\begin{figure}[!htbp]
    \centering

\begin{tikzpicture}[
    dot/.style={circle, fill=white, draw=black, inner sep=0pt, minimum size=4pt},
    line/.style={black, thick}
]
\node[dot] (A) at (-3.8, 0) {};
\node[dot] (B) at (-2.8, 0) {};
\node[dot] (C) at (-1.8, 0) {};
\node[dot] (D) at (-0.8, 0) {};
\node[label, below=1pt of A] {$u$};
\node[dot] (E) at (0.8, -0.5) {};
\node[dot] (F) at (0.8, 0.5) {};
\node[dot] (G) at (1.6, -1) {};
\node[dot] (H) at (1.6, 0) {};
\node[dot] (I) at (1.6, 1) {};
\node[dot] (J) at (2.4, -0.5) {};
\node[dot] (K) at (2.4, 0.5) {};
\node[label, below=1pt of H] {$u$};

\draw[line] (A) -- (B) -- (C) -- (D);
\draw[line] (E) -- (F) -- (I) -- (K) -- (J) -- (G) -- (E);
\draw[line] (H) -- (I);
\draw[line] (H) -- (E);
\draw[line] (H) -- (J);
\end{tikzpicture}
\caption{$P_4$ and $Q_3^-$.}
    \label{fig:example}
\end{figure}

The paper is organized as follows. In the next section, we recall relevant definitions and establish the notation that will be used throughout the paper.
Then, we prove Theorems \ref{the01}, \ref{the02} and \ref{the04} in Section~3. Finally, Section~4 closes the paper with remarks and open questions.

\section{Preliminaries}
Throughout this paper, unless stated otherwise, all graphs considered are undirected, finite, and simple. Let $G$ be a graph with vertex set $V(G)$ and edge set $E(G)$. If all pairs of vertices of a subgraph $H$ of $G$ that are adjacent in $G$ are also adjacent in $H$,
then $H$ is an \textit{induced subgraph}. We write $G[X]$ for the subgraph of $G$ induced by $X \subseteq V(G)$. 
The \textit{interval} $I_G(u,v)$ between two vertices $u$ and $v$ of a graph $G$ is the set of vertices on the shortest paths between $u$ and $v$. For $u,v\in V(G)$, the {\em distance} between $u$ and $v$ is defined as the length of a shortest $u,v$-path, denoted by $d_G(u,v)$. 
For a subgraph $H$, if $d_H(u,v) =d_G(u,v)$ for all $u,v \in V(H)$, we say $H$ is an \textit{isometric subgraph}.  
A graph $G$ is called {\em partial cube} if it is isomorphic to an isometric subgraph of $Q_n$ for some $n$.  When there is no ambiguity, we will drop the subscript $G$ of $I_G(u,v)$ and $d_G(u,v)$.

$G$ is called a {\em median graph} if for every three different vertices $u,v,w\in V(G)$, there exists exactly one vertex $x\in V(G)$ (maybe $x\in\{u,v,w\}$), called the {\em median} of $u,v,w$, satisfying that $d(u,x)+d(x,v)=d(u,v)$, $d(u,x)+d(x,w)=d(u,w)$ and $d(v,x)+d(x,w)=d(v,w)$.

For a graph $G$, the {\em simplex graph} $S(G)$, introduced by Bandelt and van de Vel \cite{bv89}, is the graph whose vertices are the cliques of $G$ (including the empty set), with two vertices being adjacent if, as cliques of $G$, they differ in exactly one vertex. Assume $V(G)=[n]$. Let $\mathcal{K}(G)$ be the set of clique of $G$, i.e., $\mathcal{K}(G)=V(S(G))$. Recall the one-to-one correspondence $\psi:\mathcal{K}(G)\to B^n$, which is defined as follows: For $u\in \mathcal{K}(G)$ and $\psi(u):=v=v_1v_2\cdots v_n\in B^n$, $i\in u\iff v_i=1$. For convenience, in the following text, we denote the vertex $u\in\mathcal{K}(G)$ as $\psi(u)$. 

For $u\in B^n$, the \textit{weight} of $u$ is the number of 1s in word $u$, denoted as $w(u)$. If $H$ is an induced $k$-cube in $Q_n$, it is well known that $H$ contains a unique vertex of maximum (resp.\ minimum) weight; see~\cite{SKM}.
More generally, if $H$ is an induced $k$-cube in a partial cube $G$, then for any $u\in V(G)$ there exists a unique vertex in $H$ at maximum (resp.\ minimum) distance to $u$.
We refer to this vertex as the {\em top vertex} (resp.\ {\em base vertex}) of $H$ with respect to the vertex $u$. 

The {\em Djokovi\'c-Winkler relation} (see \cite{dj73,w84}) $\Theta$ is a binary relation on $E(G)$ defined as follows: Let $e=uv$ and $f=xy$ be two edges in $G$, $e\,\Theta\,f\iff d(u,x)+d(v,y)\neq d(u,y)+d(v,x)$. Winkler \cite{w84} proved that a graph $G$ is a partial cube if and only if $G$ is bipartite and $\Theta$ is an equivalence relation on $E(G)$. 

Let $G$ be a partial cube. We call the equivalence class on $E(G)$ {\em $\Theta$-class}. For $e=uv\in E(G)$, we denote the $\Theta$-class containing $uv$ as $F_{uv}$, i.e., $F_{uv}:=\{f\in E(G)|f\,\Theta\,e\}$.  
Moreover, we denote $W_{uv}:=\{w\in V(G)|d_{G}(u,w)<d_{G}(v,w)\}$.

Let $G$ be a partial cube and not $K_1$, $F_{ab},F_{uv}$ two $\Theta$-classes of $G$. We say $F_{ab},F_{uv}$ {\em cross} if $W_{ab}\cap W_{uv}\neq\emptyset$, $W_{ba}\cap W_{uv}\neq\emptyset$, $W_{ab}\cap W_{vu}\neq\emptyset$ and $W_{ba}\cap W_{vu}\neq\emptyset$. The {\em crossing graph} of $G$ (see \cite{km02}), denote by $G^{\#}$, is the graph whose vertices are corresponding to the $\Theta$-classes of $G$, and $F_{ab},F_{uv}$ are adjacent in $G^{\#}$ if and only if they cross in $G$.

An {\em $i$-clique} of a graph $G$ is a subset of $V(G)$ of $i$ vertices whose induced subgraph is complete. Let's define $a_i(G)$ as the number of $i$-clique in $G$ for $i\geq 1$ and $a_0(G)=1$. The {\em clique polynomial} of $G$  is defined as follows, which is introduced by Hoede and Li \cite{hl94}.
\begin{equation*}
Cl_G(x):=\sum_{i\geq 0}a_i(G)x^i.
\end{equation*}

For a polynomial $P(x_1,x_2,\cdots,x_n)$, We denote the coefficient of $x_1^{d_1}x_2^{d_2}\cdots x_n^{d_n}$ in $P$ by $[x_1^{d_1}x_2^{d_2}\cdots x_n^{d_n}]P(x_1,x_2,\cdots,x_n)$.

\section{Proofs of the main results}
\subsection{Proofs of Theorems~\ref{the01} and \ref{the02}}
Before proceeding with the proofs of Theorems~\ref{the01} and \ref{the02}, we first present a surprisingly neat lemma, which provides a characterization of a daisy cube.

\begin{lemma}\label{lem03}
Let $G$ be a partial cube with $0^n\in V(G)$. Then $G$ is a daisy cube if and only if $G[I_G(0^n,v)] \cong Q_{w(v)} $ for all $v \in V(G)$.
\end{lemma}

\begin{proof}
{\em Necessity (\(\Rightarrow\)).}
Since $G$ is a daisy cube, for each $v \in V(G)$, any binary string $z$ with $z\preceq v$ is contained in $V(G)$. Consequently, $I_G(0^n,v)=\{z:z\preceq v\}$. Thus, $G[I_G(0^n,v)] \cong Q_{w(v)}.$

{\em Sufficiency (\(\Leftarrow\)).}
If \(G\) is not a daisy cube, then there exists \(v \in V(G)\) such that some \(z \preceq v\) satisfies $z \notin V(G)$.
Consequently, $z \notin I_G(0^n, v)$, but $z$ is a vertex in $Q_{w(v)}$.
Hence, $ I_G(0^n, v)  \neq V(Q_{w(v)})$, a contradiction.
\end{proof}

\medskip
Now, we prove Theorems \ref{the01} and \ref{the02}.
\begin{proof}[\textbf{Proof of Theorem \ref{the01}}]
We establish the equivalence by verifying the necessary implications below.

(i)$\implies$(ii): The result follows from Proposition \ref{prop01}.
\smallskip

(ii)$\implies$(i):  By contradiction assume that $G$ is not a daisy cube. 

By the definition, $[x]C_G(x)=|E(G)|$. For $v\in V(G)$, set 
\[
e(v):=|\{e=vw\in E(G): d(0^n, v)>d(0^n, w)\}|.
\]
Then $|E(G)|=\sum\limits_{v\in V(G)}e(v)$ by the bipartiteness of $G$. Therefore
\[
[x]C_G(x)=\sum\limits_{v\in V(G)}e(v).
\] 

Now, we consider $[x]W_{G, 0^n}(x+1)$. By calculation, $[x]W_{G,0^n}(x+1)=\sum\limits_{d\geq 0}dw_d(G)$. Recall that $w_d(G)=|\{v\in V(G): d(v, 0^n)=d\}|$. So 
$[x]W_{G, 0^n}(x+1)=\sum\limits_{d\geq 0}dw_d(G)=\sum\limits_{v\in V(G)}d(0^n, v)=\sum\limits_{v\in V(G)}w(v)$. 
Note that the last equality holds because $G$ is a particial cube and recall $w(v)$ is the number of 1 in $v$. 

Since $G$ is an induced subgraph of $Q_n$, $e(v)\leq w(v)$ for any $v\in V(G)$. Again since $G$ is not a daisy cube, there exists a vertex $v$ satisfying  $e(v)<w(v)$. Therefore $[x]C_G(x)<[x]W_{G,0^n} (x+1),$ a contradiction.

\smallskip

(i)$\implies$(iii): The result follows from Proposition \ref{prop01}.
\smallskip

(iii)$\implies$(iv): By (iii), we have $D_{G,0^n}(x,y)=W_{G,0^n}(x+y)$ and $D_{G,0^n}(x+y-1, 1)=W_{G,0^n}(x+y)$. Therefore $D_{G,0^n}(x,y)=D_{G,0^n}(x+y-1, 1)$. Combined with Eq. \eqref{eq:D=C}, we have $D_{G,0^n}(x,y)=C_{G}(x+y-1)$, i.e., Eq. (iv), as desired.
\smallskip

(iv)$\implies$(i): Assume to the contrary that $G$ is not a daisy cube. By Lemma \ref{lem03}, there exists  $z \in V(G)$ (set $l:=w(z)$) such that $G[I_G(0^n, z)]$ $ \not\cong Q_{l}$ . We consider $[x^l]D_{G,0^n}(x,y)$ and $[y^l]D_{G,0^n}(x,y)$. By the definition of $D_{G,0^n}(x,y)$, $[x^l]D_{G,0^n}(x,y)(=[x^ly^0]D_{G,0^n}(x,y))$ is the number of the induced subgraphs of $G$ isomorphic to $Q_l$ containing vertex $0^n$, namely,  the size of $S:=\{v \in V(G)|w(v)=l\mbox{ and }G[I_G(0^n, v)]= Q_{l}\}$. Similarly,  $[y^l]D_{G,0^n}(x,y)(=[x^0y^l]D_{G,0^n}(x,y))$ is the number of vertices at distance $l$ from $0^n$, i.e., the size of $T:=\{v \in V(G)|w(v)=l\}$. Obviously, $S\subseteq T$ and $z\in T\setminus S$.  So we have $S\subsetneq T$. Thus, 

\begin{align}
[x^l]D_{G,0^n}(x, y)<[y^l]D_{G,0^n}(x, y). \label{eq:Dx<Dy}
\end{align}

However, by Eq. (iv), 
\begin{align}
[x^l]D_{G, 0^n}(x,y)=[x^l]C_G(x+y-1), \label{eq:Dx=Cx}
\end{align}
and 
\begin{align}
    [y^l]D_{G, 0^n}(x,y)=[y^l]C_G(x+y-1). \label{eq:Dy=Cy}
\end{align}
While, in $C_G(x+y-1)$,
\begin{align}
    [x^l]C_G(x+y-1)=[y^l]C_G(x+y-1). \label{eq:Cx=Cy}
\end{align}
Combined with Eqs. \eqref{eq:Dx=Cx}, \eqref{eq:Dy=Cy} and \eqref{eq:Cx=Cy}, we have 
\[[x^l]D_{G, 0^n}(x,y)=[y^l]D_{G, 0^n}(x,y),
\]
which is contradicted with Ineq. \eqref{eq:Dx<Dy}. Therefore, $G$ is a daisy cube.
\end{proof}

\begin{proof}[\textbf{Proof of Theorem \ref{the02}}]
The case of equality is an immediate consequence of Theorem \ref{the01}. It is sufficient to show that $D_{G, u}(x, y) \leq W_{G,u}(x+y)$.

By the definition of $D_{G, u}(x, y)$, $[x^ky^d]D_{G, u}(x, y)(=c_{k, d}(G))$ is the size of the set $A$ of $k$-cubes at distance $d$ from $u$ in $G$, one-to-one corresponding to the set of ordered pairs $(t, b)$, where $t, b$ are the top vertex and the base vertex of a $k$-cube in $A$, respectively. Note that, here $t, b$ satisfy that $d(u, b)=d, d(b, t)=k, d(u, t)=d(u, b)+d(b, t)=d+k$. For every top vertex $t$, there are no more than $\binom{d+k}{k}$ base vertices $b$ satisfying the above three equality, while the number of $t$ is no more than $w_{d+k}(G)$. Thus $|A|\leq w_{d+k}(G)\binom{d+k}{k}$, i.e., 
\begin{align}
c_{k, d}(G)\leq w_{d+k}(G)\binom{d+k}{k}. \label{eq:c(k,d)<wBinom}
\end{align}
Substituting Ineq. \eqref{eq:c(k,d)<wBinom} into $D_{G, u}(x, y)$, we obtain
\begin{align*}
    D_{G, u}(x, y)&\leq \sum_{k, d\geq 0} w_{d+k}(G)\binom{d+k}{k} 
    x^ky^d\\
    &=\sum_{d+k \geq 0}w_{d+k}(G)(x+y)^{d+k}\\
    &=W_{G, u}(x+y),
\end{align*}
as desired.
\end{proof}

\subsection{Proof of Theorem \ref{the04}}

Before presenting the proof of Theorem \ref{the04}, we need to introduce some terminology and lemmas, which originated from \cite{xfx24}.

Let $G$ be a graph. A set $A\subseteq V(G)$ is said to be {\em convex} if for all $u, v\in A$ the set $A$ contains all the vertices lied on a shortest path between $u$ and $v$ in $G$. Note that the intersection of  a finite number of convex sets is also convex. The {\em convex hull} $co_G(S)$ of a set $S\subseteq V(G)$ is defined as the smallest convex set containing $S$. Let $\mathcal{C}(G)$ be the set of all isometric cycles of $G$. We define a binary relation `$\leq_{\mathcal{C}(G)}$' on $\mathcal{C}(G)$ as follows: For any $C,C'\in\mathcal{C}(G)$, $C\leq_{\mathcal{C}(G)}C'\iff co_G(V(C))\subseteq co_G(V(C'))$. It's easy to see that `$\leq_{\mathcal{C}(G)}$' is a partial order on $\mathcal{C}(G)$. The set of maximal isometric cycles in $G$ is denoted by $\mathcal{C}_{\max}(G)$.

For a partial cube $G$ embedded into $Q_n$, we construct a new graph $G^+$ as described in Algorithm 1, which is adapted from \cite{xfx24}. Note that in this case if $C\in \mathcal{C}(G)$, then  $C\in \mathcal{C}(Q_n)$. It  is obvious that, being a partial cube, $G$ is an induced subgraph of $G^+$. About $G^+$, we have the following properties.

\begin{lemma}{\em\cite{xfx24}}\label{lem:G+Cro}
Let  $G(\neq K_1)$ be a partial cube, $G^+$ the graph constructed from Algorithm 1. Then
\begin{itemize}
\item[(1)] $G^{+}$ is a median graph;
\item[(2)] $G^{\#}=(G^+)^{\#}$.
\end{itemize}
\end{lemma}

\begin{algorithm}%[h]
  \caption{Constructing $G^+$  from a partial cube $G$}
  \begin{algorithmic}[1]
    \Require a partial cube $G$ embedded into $Q_n$ 
    \Ensure  a new graph $G^+$
    \State $G^{(0)}\leftarrow G$.
    \State $S^{(0)}\leftarrow V(G)$.
   \State $i \leftarrow 0$.
    \While{$i\geq 0$}
    	\State Compute $\mathcal{C}_{\max}(G^{(i)})$. 
    	\State $S_{i+1}\leftarrow S_i\cup\bigcup\limits_{C\in\mathcal{C}_{\max}(G^{(i)})}co_{Q_n}(V(C))$.
    	\If{$S_{i+1}=S_i$}
    		\State $G^+\leftarrow G^{(i)}$.
    		\Break
    	\EndIf
    	\State $G^{(i+1)}\leftarrow Q_n[S_{i+1}]$.
    	\State $i\leftarrow i+1$.
    \EndWhile
\end{algorithmic}
\end{algorithm}

Now, we prove Theorem \ref{the04}.
\begin{proof}[\textbf{Proof of Theorem \ref{the04}}]
By the discussion before Theorem \ref{the04}, it is sufficient to prove the sufficiency, i.e., to prove that if $G$ is a partial cube embedded into $Q_n$ with $0^n\in V(G)$ satisfying $W_{G,0^n}(x)=Cl_{G^{\#}}(x)$, then $G$ is a simplex graph. By Lemma \ref{lem04}, it suffices to show that $G$ is both a median graph and a daisy cube.

\begin{claim*}
    $G$ is a median graph.
\end{claim*} 
Let $G^{+}$ be the graph from $G$ as presented in Algorithm 1. Combined Theorem \ref{prop02} and Lemma \ref{lem:G+Cro} (1), we have  
\begin{align}
    C_{G^+}(x)=Cl_{(G^+)^{\#}}(x+1). \label{eq:C+=Cl+}
\end{align}
By Lemma \ref{lem:G+Cro} (2),
\begin{align}
    Cl_{(G^+)^{\#}}(x+1)=Cl_{G^{\#}}(x+1). \label{eq:G+=G(Cl)}
\end{align}
Combined Eqs. (\ref{eq:C+=Cl+}), (\ref{eq:G+=G(Cl)}) with Condition $W_{G,0^n}(x)=Cl_{G^{\#}}(x)$,
\begin{align}
    C_{G^+}(x)=W_{G,0^n}(x+1). \label{eq:C+=W}
\end{align}
Substituting $x=0$ into Eq. (\ref{eq:C+=W}) gives 
$$C_{G^+}(0)=W_{G,0^n}(1).$$
By the definitions, $C_{G^+}(0)$ is the number of vertices in $G^+$, while $W_{G,0^n}(1)$ is the number of vertices in $G$. Combined with $G$ being an induced subgraph of $G^+$, $G=G^+$. Thus, $G$ is a median graph.
\begin{claim*}
    $G$ is a daisy cube.
\end{claim*}
Since $G$ is a median graph, $C_G(x)=Cl_{G^{\#}}(x+1)$ by Theorem \ref{prop02}. Combined with $Cl_{G^{\#}}(x+1)=W_{G,0^n}(x+1)$, we have $C_G(x)=W_{G,0^n}(x+1)$, which implies $G$ is a daisy cube by Theorem \ref{the01}.

Therefore, $G$ is a simplex graph by Lemma \ref{lem04}.
\end{proof}

\section{Conclusion and some remarks}
In the present paper, we obtained three equivalent characterizations of a daisy cube $G$ via any one of the related polynomial identities:
\begin{enumerate}
    \item[(a)] $C_{G}(x)=W_{G, 0^n}(x+1)$;
    \item[(b)] $D_{G, 0^n}(x, y)=W_{G, 0^n}(x+y)$;
    \item[(c)] $D_{G, 0^n}(x, y)=C_{G}(x+y-1)$.
\end{enumerate}
Moreover, we proved that inequalities
\begin{enumerate}
    \item[(d)]$C_{G}(x)\leq W_{G, u}(x+1)$
and
    \item[(e)] $D_{G, u}(x, y)\leq W_{G, u}(x+y)$
\end{enumerate}
hold for a partial cube $G$. Finally, we proved that the two upper bounds,  $W_{G, u}(x+1)$ (from inequality (d)) and $Cl_{G^{\#}}(x+1)$ (from  Ineq. \eqref{ine:crossing}), on $C_G(x)$ for a partial cube $G$, coincide if and only if $G$ is a simplex graph with $u=0^n$.

Compared (a), (b), (c) with (d), (e), a natural question arises: does an inequality between
$D_{G,u}(x,y)$ and $C_{G}(x+y-1)$ hold for partial cubes? We give a negative answer. 
For example, Let $G$ be the partial cube illustrated in Figure \ref{fig:counterexample}.
\begin{align*}
D_{G,u}(x,y) &= 3x^2 + 6xy + 3y^2 + 4x + 4y + 2y^3 + 4xy^2 + 2x^2y + 1, \\
C_{G}(x+y-1) &= 5x^2 + 10xy + 5y^2 + 4x + 4y + 1.
\end{align*}
In fact, the above two polynomials are incomparable.

\begin{figure}[!htbp]
    \centering

\begin{tikzpicture}[
    dot/.style={circle, fill=white, draw=black, inner sep=0pt, minimum size=4pt},
    line/.style={black, thick}
]

\node[dot] (A) at (-1.8, 1) {};
\node[dot] (B) at (-2.8, 0) {};
\node[dot] (C) at (-1.8, -1) {};

\node[dot] (D) at (-1, 0) {};

\node[dot] (E) at (0, 1) {};
\node[dot] (F) at (0, -1) {};

\node[dot] (G) at (1, 0) {};

\node[dot] (H) at (1.8, 1) {};
\node[dot] (I) at (2.8, 0) {};
\node[dot] (J) at (1.8, -1) {};

\node[label, below=1pt of F] {$u$};

\draw[line] (A) -- (B) -- (C);
\draw[line] (H) -- (I) -- (J);
\draw[line] (A) -- (E) -- (H);
\draw[line] (C) -- (F) -- (J);
\draw[line] (E) -- (D) -- (F);
\draw[line] (E) -- (G) -- (F);
\draw[line] (B) -- (D);
\draw[line] (G) -- (I);

\end{tikzpicture}
\caption{The partial cube $G$.}
    \label{fig:counterexample}
\end{figure}

According to Algorithm 1, we can obtain $G^+$ from a partial cube $G$ such that ($1$) $G^{+}$ is isomorphic to a median graph; ($2$) $G$ is an induced subgraph of $G^{+}$; ($3$) if $G$ is a median graph, then $G^{+}=G$, and further,
$$C_G(x)\leq C_{G^{+}}(x),$$
and the equality holds if and only if $G$ is a median graph.

Similarly, for a partial cube $G$ embedded in $Q_n$ and $u$ a given vertex in $V(G)$, we can also construct a graph $G'_u:=Q_n\left[\bigcup\limits_{v\in V(G)}I_{Q_n}(u,v)\right]$. It is obvious that ($1'$) $G'_u$ is isomorphic to a daisy cube; ($2'$) $G$ is an induced subgraph of $G'_u$; ($3'$) if $G$ is a daisy cube with $u=0^n$, then $G'_u=G$. Thus, for any partial cube $G$ with $u\in V(G)$,
$$C_G(x)\leq C_{G'_u}(x),$$
and the equality holds if and only if $G$ is a daisy cube and $u=0^n$.

A natural question arises: given a partial cube $G$, does there exist a graph $G^*$ obtained by adding fewer vertices to $G$ than in the constructions of the graph $G'_u$ or $G^+$, such that $G^*$ lies in a larger subclass of partial cubes containing both median graphs and daisy cubes? Before addressing this question, we first introduce a class of partial cubes: tope graphs of lopsided sets.

Let $X$ be a finite set. Denote $\mathrm{Sign}(X):=\{-1,1\}^X$, that is, the set of all maps from $X$ into $\{-1,1\}$. Let $\mathcal{S}$ be a subset of $\mathrm{Sign}(X)$ and $Y$ a subset of $X$. For $t\in\mathrm{Sign}(X-Y)$ and $s\in\mathrm{Sign}(X)$, if $s(e)=t(e)$ for all $e\in X-Y$, we say that $s$ is an {\em extension} of $t$. Denote 
$$\mathcal{S}^Y:=\{t\in \mathrm{Sign}(X-Y)|\mbox{every extension of }t\mbox{ belongs to }\mathcal{S}\},$$
and further,
$$\underline{\mathcal{X}}(\mathcal{S}):=\{Y\subseteq X|\mathcal{S}^Y\neq\emptyset\}.$$
Using this notation, $\mathcal{S}$ is called a {\em lopsided set (LOP)} if for every $A\subseteq X$, either $A\in\underline{\mathcal{X}}(\mathcal{S})$, or $X-A\in\underline{\mathcal{X}}(\mathrm{Sign}(X)-\mathcal{S})$ \cite{bcdk06,l83}.
The {\em tope graph} of an LOP $\mathcal{S}$ on a finite set $X$ is the graph whose vertex set is $\mathcal{S}$. Two vertices $f,g\in\mathcal{S}$ are adjacent if there exists exactly one $e\in X$ such that $f(e)\neq g(e)$ \cite{km20}. 
The tope graphs of LOPs are known to be partial cubes and the superclass of median graphs and daisy cubes \cite{bcdk06,l83,m18}.

With this definition of the tope graph of LOP in hand, we can now state the problem mentioned earlier.
\begin{problem}\label{prob02}
For any partial cube $G$, provide an algorithm for constructing a tope graph $G^*$ of an LOP  such that $G$ is an induced subgraph of $G^*$, and $G^*=G$ if and only if $G$ is a tope graph of an LOP.
\end{problem}

Finally, we consider another question. First, we give a definition. For two graphs $G$ and $H$, if $C_G(x)=C_H(x)$, we say $G$ and $H$ are  \textit{$C$-equivalent}. We would like to ask:
\begin{question}\label{que:Cequivalent}
For a given graph class $\mathscr{C}$, does there exist a subclass $\mathscr{C}'\subset \mathscr{C}$ such that for every $G\in\mathscr{C}$, there is $G'\in \mathscr{C}'$ such that $G$ is $C$-equivalent to $G'$?
\end{question}

We claim that if $\mathscr{C}$ is the class of median graphs, the answer of above question is `yes' and $\mathscr{C}'$ is the class of simplex graphs in this case.
\begin{theorem}\label{S=M}
For any median graph $G$, there exists a simplex graph $G'$ such that $C_G(x)=C_{G'}(x)$.
\end{theorem} 
 Before we proceed to prove Theorem \ref{S=M}, we first present two lemmas that will be used in the proof.
\begin{lemma}[\cite{xx25}]\label{lem:S(G)=G}
    For any graph $G$, $G=(S(G))^\#.$
\end{lemma}
\begin{lemma}[\cite{xx25}]\label{lem:C=C}
    Let $G$, $H$ be two median graphs and $G, H\neq K_1$. If $G^\#\cong H^\#$, then
    $C_G(x)=C_H(x).$
\end{lemma}
With these lemmas in hand, we  provide the proof of Theorem \ref{S=M}.

\begin{proof}[Proof of Theorem \ref{S=M}]
Let $G$ be a median graph.  By Lemma \ref{lem:S(G)=G}, 
$$G^{\#}=\left(S\left(G^{\#}\right)\right)^{\#}.$$
Set $G'=S\left(G^{\#}\right)$,  namely, a simplex graph  (therefore of  a median graph). 
%Then $$G^{\#}=\left(G'\right)^{\#}.$$
By Lemma \ref{lem:C=C}, we obtain $C_G(x)=C_{G'}(x)$.
\end{proof}

By Lemma \ref{lem04}, the class of simplex graphs is the intersection of median graphs and daisy cubes. If $\mathscr{C'}$ in Question \ref{que:Cequivalent} is the class of daisy cubes, can we obtain a generalized version of Theorem \ref{S=M}? For example, we can propose the following problem.

\begin{problem}\label{pro02}
For any tope graph $G$ of LOP, does there exist a daisy cube $G'$ satisfying that
$$C_G(x)=C_{G'}(x)?$$
\end{problem}

\end{document}